\documentclass[a4paper, 11pt]{amsart}
\pdfoutput=1

\usepackage[utf8]{inputenc}
\usepackage[T1]{fontenc}
\usepackage{lmodern}
\usepackage[english]{babel}
\usepackage[dvipsnames]{xcolor}
\usepackage{adjustbox}
\usepackage{letltxmacro}
\usepackage{todonotes}
\LetLtxMacro\todonotestodo\todo
\renewcommand{\todo}[2][]{\todonotestodo[#1]{TODO: {#2}}}
\usepackage{enumerate}
\usepackage{hyperref}
\usepackage{url}
\usepackage{adjustbox}

\usepackage{tikz}
\usetikzlibrary{matrix,decorations.pathreplacing,calc}
\usetikzlibrary{decorations.pathmorphing, patterns,shapes, snakes}
\usetikzlibrary{patterns}

\pgfkeys{tikz/mymatrixenv/.style={decoration=brace,every left
    delimiter/.style={xshift=3pt},every right
    delimiter/.style={xshift=-3pt}}}

\pgfkeys{tikz/mymatrix/.style={matrix of math nodes,left
    delimiter=(,right delimiter={)},inner sep=6pt,column sep=1em,row
    sep=0.5em,nodes={inner sep=0pt}}}

\pgfkeys{tikz/mymatrixbrace/.style={decorate,thick}}
\newcommand\mymatrixbraceoffseth{0.7em}
\newcommand\mymatrixbraceoffsetv{0.7em}

\newcommand*\mymatrixbraceleft[4][m]{
    \draw[mymatrixbrace] ($(#1.north west)!(#1-#3-1.south west)!(#1.south west)-(\mymatrixbraceoffseth,1pt)$)
        -- node[left=2pt] {#4} 
        ($(#1.north west)!(#1-#2-1.north west)!(#1.south west)-(\mymatrixbraceoffseth,-6pt)$);
}
\newcommand*\mymatrixbraceright[4][m]{
    \draw[mymatrixbrace] ($(#1.north east)!(#1-#2-1.north east)!(#1.south east)+(\mymatrixbraceoffseth,2pt)$)
        -- node[right=2pt] {#4} 
        ($(#1.north east)!(#1-#3-1.south east)!(#1.south east)+(\mymatrixbraceoffseth,-2pt)$);
}
\newcommand*\mymatrixbracetop[4][m]{
    \draw[mymatrixbrace] ($(#1.north west)!(#1-1-#2.north west)!(#1.north east)+(-1pt,\mymatrixbraceoffsetv)$)
        -- node[above=2pt] {#4} 
        ($(#1.north west)!(#1-1-#3.north east)!(#1.north east)+(2pt,\mymatrixbraceoffsetv)$);
}
\newcommand*\mymatrixbracebottom[4][m]{
    \draw[mymatrixbrace] ($(#1.south west)!(#1-1-#3.south east)!(#1.south east)-(0,\mymatrixbraceoffsetv)$)
        -- node[below=2pt] {#4} 
        ($(#1.south west)!(#1-1-#2.south west)!(#1.south east)-(0,\mymatrixbraceoffsetv)$);
}

\newtheorem{theorem}{Theorem}

\makeatletter
\newtheorem*{rep@theorem}{\rep@title}
\newcommand{\newreptheorem}[2]{%
\newenvironment{rep#1}[1]{%
 \def\rep@title{#2 \ref{##1}}%
 \begin{rep@theorem}}%
 {\end{rep@theorem}}}
\makeatother

\newreptheorem{theorem}{Theorem}

\newtheorem{lemma}{Lemma}[section]
\newtheorem{proposition}[lemma]{Proposition}
\newtheorem{corollary}[lemma]{Corollary}
\newtheorem{fact}[lemma]{Fact}

\theoremstyle{definition}
\newtheorem{remark}[lemma]{Remark}

\newtheorem{definition}[lemma]{Definition}
\newtheorem*{claim*}{Claim}
\newtheorem{notation}[lemma]{Notation}

\newcommand{\A}{\mathbf{A}}
\newcommand{\B}{\mathbf{B}}

\newcommand{\Q}{\mathbb{Q}}
\newcommand{\Z}{\mathbb{Z}}
\newcommand{\N}{\mathbb{N}}
\newcommand{\val}{\mathbf{v}}
\newcommand{\calP}{\mathcal{P}}
\newcommand{\calJ}{\mathcal{J}}
\newcommand{\calR}{\mathcal{R}}

\newcommand{\rempn}{r_{n,p}}

\DeclareMathOperator{\rank}{rank}
\DeclareMathOperator{\spn}{span}
\DeclareMathOperator{\Int}{Int}
\DeclareMathOperator{\inspn}{\mathcal{I}}
\DeclareMathOperator{\out}{\mathcal{O}}
\newcommand{\outspn}[2]{\out_{[#1,#2]}}
\newcommand{\pblock}[2]{\B_{#1,#2}}

\title{Absolute irreducibility of the binomial polynomials}
\author{Roswitha Rissner}
\address[Roswitha Rissner]{Institut für Mathematik\\Alpen-Adria-Universität Klagenfurt\\
  Universitätsstraße 65-67\\9020 Klagenfurt am Wörthersee\\Austria}
\email{\href{mailto:roswitha.rissner@aau.at}{roswitha.rissner@aau.at}}
\thanks{R.~Rissner is supported by the Austrian Science Fund (FWF): P~28466}

\author{Daniel Windisch}
\address[Daniel Windisch]{Institut für Analysis und Zahlentheorie\\Technische Universität Graz\\
  Kopernikusgasse 24/II\\8010 Graz\\Austria}
\email{\href{mailto:dwindisch@math.tugraz.at}{dwindisch@math.tugraz.at}}
\thanks{D.~Windisch is supported by the Austrian Science Fund (FWF): P~30934}

\begin{document}

\maketitle

\begin{abstract}
  In this paper we investigate the factorization behaviour of the
  binomial polynomials
  $\binom{x}{n} = \frac{x(x-1)\cdots (x-n+1)}{n!}$ and their powers
  in the ring of integer-valued polynomials $\Int(\Z)$.
%
  While it is well-known that the binomial polynomials are irreducible
  elements in $\Int(\Z)$,
%
  the factorization behaviour of their powers has not yet been fully
  understood.
%
  We fill this gap and show that the binomial polynomials are
  absolutely irreducible in $\Int(\Z)$, that is, $\binom{x}{n}^m$
  factors uniquely into irreducible elements in $\Int(\Z)$ for all
  $m\in \N$.
%
  By reformulating the problem in terms of linear algebra and number
  theory, we show that the question can be reduced to determining the
  rank of, what we call, the valuation matrix of $n$.  A main
  ingredient in computing this rank is the following
  number-theoretical result for which we also provide a proof: If
  $n>10$ and $n$, $n-1$, \ldots, $n-(k-1)$ are composite integers,
  then there exists a prime number $p > 2k$ that divides one of these
  integers.

\end{abstract}

\section{Introduction}
\label{section:introduction}

In this work, our main objects of interest are the
so-called \textit{binomial polynomials} 
\begin{align*}
\binom{x}{n}  = \frac{x(x-1) \cdots (x-n+1)}{n!}
\end{align*}
for integers $n\ge 1$ as elements of
$\Int(\Z) = \{ f \in \Q[x] \mid f(\Z)\subseteq \Z \}$, where $\Z$
denotes the ring of integers and $\Q$ is the field of rational
numbers.

It is well-known that $n!$ divides the product of any $n$ consecutive
integers, and therefore $\binom{x}{n}$ indeed is an integer-valued
polynomial for all $n \geq 1$. Moreover, the binomial polynomials are
known to be irreducible in $\Int(\Z)$, cf.~the survey of Cahen and
Chabert~\cite{Cahen-Chabert:2016:survey}. A non-zero non-unit $a$ in a
(commutative) integral domain $D$ is said to be \textit{irreducible}
if $a = bc$ implies that either $b$ or $c$ is a unit for all
$b$, $c \in D$.

However, the irreducibility of $a$ does not tell anything about the
factorization behaviour of the powers $a^m$.  An irreducible element
$a$ is said to be \textit{absolutely irreducible} (or a \textit{strong
  atom}) if $a^m$ factors uniquely into irreducible elements for all
integers $m \geq 1$. There are plenty of non-absolutely irreducible
elements in $\Int(\Z)$, as it was shown lately by
Nakato~\cite{Nakato:2020:non-absolutely}. Moreover, the recent work of
Frisch and Nakato~\cite{Frisch-Nakato:2020:graph-theoretic} gives a
criterion for the absolute irreducibility of integer-valued
polynomials with square-free denominators over $\Z$. In particular, it
follows from their Theorem~2 that, for all prime numbers $p$, the
binomial polynomial $\binom{x}{p}$ is absolutely irreducible
(see~\cite[Example~2.6]{Frisch-Nakato:2020:graph-theoretic}). That the
latter holds was shown before by
McClain~\cite{McClain:2004:honorsthesis}.

In the present paper we show that the binomial polynomial
$\binom{x}{n} \in \Int(\Z)$ is absolutely irreducible for all
$n\geq 1$. Our approach also covers the case that $n=p$ is a prime
number, and hence serves in this special case also as an alternative
proof to the one in~\cite{Frisch-Nakato:2020:graph-theoretic}.

The binomial polynomials play a central role in the study of
$\Int(\Z)$ as they form a so-called regular $\Z$-module basis, that
is, a basis which contains exactly one polynomial of each degree.
Implicitly, this fact was already applied by Newton who used
integer-valued polynomials to interpolate integer-valued functions on
$\Z$.

Our viewpoint on the binomial polynomials in this paper relates to the
whole area of investigating non-unique factorizations.  When
mathematicians first explored that factorizations into irreducible
elements do not have to be unique in general, they considered this
behaviour as pathological and passed over to unique factorizations of
ideals into prime ideals in Dedekind domains. It is a movement of the
last few decades that non-unique factorizations are viewed in their
own right. Since then, the machinery for their investigation has been
developed mostly in the direction of Krull domains and monoids,
including for instance the concept of the divisor class group of a
Krull domain or monoid. For an introduction to this topic, we refer to
the textbook of Geroldinger and Halter-Koch~\cite{GHK}.

However, almost nothing is known in the area of non-unique
factorizations in non-Noetherian Prüfer domains.  In this context,
general rings of integer-valued polynomials
\begin{align*}
\Int(D) = \{f \in K[x] \mid f(D) \subseteq D \}
\end{align*}
where $D$ is an integral domain with quotient field $K$ are of
interest. The integral domain $\Int(D)$ has been studied very
intensively during the last decades and is a standard source for
examples and counterexamples. For instance, it is well-known that
$\Int(D)$ is a non-Noetherian Prüfer domain provided that $D$ is a
Dedekind domain with finite residue fields.  A profound introduction
to the theory of integer-valued polynomial is given in the textbook
of Cahen and Chabert~\cite{Cahen-Chabert:1997:book}, and a more recent
survey is their work~\cite{Cahen-Chabert:2016:survey}.

Coming back to factorization theory, Cahen and
Chabert~\cite{Cahen-Chabert:1995:iv-full-elasticity} studied
elasticity in rings of integer-valued polynomials and proved that
$\Int(D)$ has infinite elasticity for broad classes of domains $D$. In
a joint publication, Anderson, Cahen, Chapman, Scott and
Smith~\cite{Anderson-Cahen-Chapman-Scott-Smith:1995:factorization-iv}
expanded this investigation and further studied factorization
properties such as atomicity and elasticity in rings of
integer-valued polynomials. Strengthening previous results on infinite
elasticity, Chapman and
McClain~\cite{Chapman-McClain:2005:iv-full-elasticity} proved that
every rational number $r\ge 1$ is the elasticity of a polynomial $f$
that is integer-valued on an infinite subset $S$ of a unique
factorization domain $D$, that is, $r = \frac{L}{\ell}$ where $L$ and
$\ell$ are the maximum and mininum lengths, respectively, of
factorizations of $f$. Frisch~\cite{Frisch:2013:lengths} proved that
in $\Int(\Z)$ arbitrary sets of lengths can be realized. These results
were generalized by Frisch, Nakato and
Rissner~\cite{Frisch-Nakato-Rissner:2019:lengths} to rings of
integer-valued polynomials on Dedekind domains with infinitely many
maximal ideals which are all of finite index.  Both papers use a very
specific type of irreducible elements of $\Int(D)$ to realize the
aimed sets of lengths.

On the one hand, recognizing irreducible elements in $\Int(D)$ is in
general far from being trivial. A partial answer for $\Int(\Z)$ using
an algorithmic approach can be found
in~\cite{Antoniou-Nakato-Rissner:2018:table}. On the other hand, the
behavior of products of general irreducible elements in rings of
integer-valued polynomials is not at all understood.  Our result on
the absolute irreducibility of the binomial polynomials leads to a
better understanding of one important class of irreducible elements in
$\Int(\Z)$.  It also broadens the knowledge of a different aspect of
rings of integer-valued polynomials. It has been already mentioned
that $\Int(D)$ is a non-Noetherian Prüfer domain and is therefore
non-Krull, when $D$ is a Dedekind domain with finite residue
fields. Nevertheless, Reinhart~\cite{Reinhart:2014:monadic} showed
that if $D$ is any factorial domain, then $\Int(D)$ is
\textit{monadically Krull}, i.e., the so-called \textit{monadic
  submonoid}
\begin{align*}
[\![f]\!] = \{g \in \Int(D) \mid g \text{ divides } f^n \text{ for some } n\in \N \}
\end{align*}
is a Krull monoid for each $f \in
\Int(D)$. Frisch~\cite{Frisch:2014:monadic} extended this result to
integer-valued polynomial rings on Krull domains. Moreover,
Reinhart~\cite{Reinhart:2017:classgroup} showed that for a factorial
domain $D$ the divisor class group of a monadic submonoid $[\![f]\!]$
of $\Int(D)$ is free Abelian for every $f \in D[x] \setminus \{
0\}$. For general polynomials $f \in \Int(D)$, the structure of the
divisor class group of the monadic submonoid generated by $f$ is not
known, but it can be easily seen that the following two assertions are
equivalent for an irreducible $f \in \Int(D)$:
\begin{itemize}
\item The polynomial $f \in \Int(D)$ is absolutely irreducible.
\item The monoid $[\![f]\!]$ is factorial, i.e., its divisor class group is trivial.
\end{itemize}
Therefore our main result gives a description of a whole new family of
divisor class groups of monadic submonoids of $\Int(\Z)$.

This paper is structured as follows. In Section~\ref{section:results}
we present our two main results.
Theorem~\ref{theorem:binomialpolynomials} states that the binomial
polynomials are absolutely irreducible in $\Int(\Z)$.  A second
result, Theorem~\ref{theorem:numbertheory}, is a number-theoretic
result which we need in the proof of
Theorem~\ref{theorem:binomialpolynomials} but is interesting in its
own right. It states that given a sequence of $k$ consecutive
composite integers such that the largest integer is greater than $10$,
then one of the numbers of this sequence has a prime divisor $p > 2k$.

The remaining paper is dedicated to the proofs of the main results. In
Section~\ref{section:reformulation} we explain the strategy for the
proof of Theorem~\ref{theorem:binomialpolynomials} and introduce the
necessary notation. The main idea is to rephrase the question of
absolute irreducibility of the binomial polynomial $\binom{x}{n}$ in
terms of linear algebra and number theory. For this purpose, we
introduce, what we call, the \emph{valuation matrix} $\A_n$
(Definition~\ref{definition:A}) and show that $\binom{x}{n}$ is
absolutely irreducible in $\Int(\Z)$ if $\rank(\A_n) = n-1$
(Proposition~\ref{proposition:kernelA}). This point of view also
motivates the content of Section~\ref{section:numbertheory}, where we
present our number-theoretic toolbox, including a proof of
Theorem~\ref{theorem:numbertheory}.  Finally, in
Section~\ref{section:rank} we prove that $\rank(\A_n) = n-1$ holds for
all $n\in \N$ which is the final piece for proving the absolute
irreducibility of the binomial polynomial $\binom{x}{n}$.

\section{Results}
\label{section:results}

In this section, we present the main results of this paper. The
subsequent sections are dedicated to their proofs. 

\begin{theorem}\label{theorem:binomialpolynomials}
    The binomial polynomial $\binom{x}{n}$ is absolutely irreducible in
  $\Int(\Z)$ for all $n\in\N$.
\end{theorem}
The proof of Theorem~\ref{theorem:binomialpolynomials} is a
consequence of the results of Sections~\ref{section:reformulation}
and~\ref{section:rank}. To be more specific, the theorem is the
summarized statements of Corollaries~\ref{corollary:n=P}
and~\ref{corollary:ncomposite} and Remark~\ref{remark:n=01}.

\begin{corollary}
  The monadic submonoid $[\![ \binom{x}{n} ]\!]$ of $\Int(\Z)$ is
  factorial for all $n\ge 1$. In particular, it has trivial divisor
  class group.
\end{corollary}

The theory developed in Section~\ref{section:rank} heavily depends on
some number-theoretic results which are built up in
Section~\ref{section:numbertheory}. Next to well-known facts which are
collected there, we also prove the following theorem in this section.

\begin{theorem}\label{theorem:numbertheory}
  Let $n>10$ be an integer and $P$ the maximal prime number with $P\le n$. 

  If $2\le k\le n-P$ then there exists a prime number $p>2k$ which divides
  one of the numbers $n$, $n-1$, \ldots, $n-(k-1)$.
\end{theorem}

\section{Absolute irreducibility of binomial polynomials and the
  valuation matrix}
\label{section:reformulation}

In this section we introduce the \emph{valuation matrix} $\A_n$ which
is associated to the binomial polynomial $\binom{x}{n}$ and explain
how the question of its absolute irreducibility can be answered by
determining the rank of $\A_n$. First, we discuss the notion of
absolute irreducibility.

\begin{definition}\label{definition:abs-irred}
  Let $D$ be an integral domain and $b \in D$ be an irreducible
  element. We say that $b$ is \emph{absolutely irreducible} if $b^m$
  factors uniquely into irreducible elements for every $m\in \N$.
\end{definition}

\begin{remark}\label{remark:equivalence-abs-irr}
  Let $D$ be an integral domain and $b \in D$ be an irreducible
  element. A straight forward verification shows that the following
  assertions are equivalent:
  \begin{enumerate}
  \item $b$ is absolutely irreducible.
  \item For every non-negative integer $m$ and for all $f$, $g \in D$
    with $b^m = f \cdot g$, there exist non-negative integers $k$,
    $\ell$ and units $u$, $v\in D$ such that $f = ub^k$ and
    $g = vb^\ell$.
  \end{enumerate}
\end{remark}

\begin{remark}\label{remark:n=01}
Since $x^m$ factors uniquely in $\Int(\Z)$ for
  all $m\ge 1$, it follows immediately that $\binom{x}{1} = x$ is
  absolutely irreducible. This covers the case $n=1$.
\end{remark}

\begin{remark}\label{remark:rephrase-goal}
  For the rest of this work, fix a positive integer $n\ge 2$. Our goal
  is to show that for every positive integer $m$ and $f$, $g \in \Int(\Z)$ the following property holds:
  \begin{equation*}
    \binom{x}{n}^m = f \cdot g
    \;\Longrightarrow\;f= \pm \binom{x}{n}^k \text{ and } g = \pm \binom{x}{n}^\ell \text{ with } k,\ell \in \N_0.
  \end{equation*}
  Once this is shown, it follows by
  Remark~\ref{remark:equivalence-abs-irr} that $\binom{x}{n} \in \Int(\Z)$
  is absolutely irreducible.
\end{remark}

Our first step is to give a precise description of $f$ and $g$
exploiting the fact that $\Q[x]$ is a UFD.

\begin{proposition}\label{proposition:intfactors}
  Let $n$, $m\ge 2$ and $f$, $g\in \Int(\Z)$ with
  $\binom{x}{n}^m = f\cdot g$.

  Then there exist $k_i$ and $\ell_i\in \N_0$ with $k_i + \ell_i = m$
  for $0\le i\le n-1$ such that
  \begin{equation*}
    f = \pm \prod_{i=0}^{n-1} \left(\frac{x-i}{n-i} \right)^{k_i} \text{ and }
    g = \pm \prod_{i=0}^{n-1} \left(\frac{x-i}{n-i}\right)^{\ell_i}
  \end{equation*}
  holds.
\end{proposition}

\begin{proof}
  Since both sides of the equality $\binom{x}{n}^m = f \cdot g$ factor
  uniquely in $\Q[x]$ into irreducible elements, it follows that
  \begin{equation}\label{eq:rationals}
    f = q_1\prod_{i=0}^{n-1}\left(x-i\right)^{k_i} \text{ and } g = q_2\prod_{i=0}^{n-1}
    \left(x-i\right)^{\ell_i}
  \end{equation}
  for $q_1$, $q_2 \in \Q$ and non-negative integers $k_0$, \ldots,
  $k_{n-1}$ and $\ell_0$, \ldots, $\ell_{n-1}$ with $k_i + \ell_i = m$
  for all $i \in \{0, \ldots, n-1 \}$. 

  Evaluating $\binom{x}{n}^m = f \cdot g$ at $x = n$ implies
  $f(n) \cdot g(n) = 1$. Moreover, as $f$, $g \in \Int(\Z)$, it follows
  that $f(n) = \pm 1$ and $g(n) = \pm 1$.  This observation together
  with Equations~\eqref{eq:rationals} imply that
  \begin{equation*}
    \frac{1}{q_1} = \pm \prod_{i = 0}^{n-1} (n-i)^{k_i} \text{ and } \frac{1}{q_2} = \pm \prod_{i = 0}^{n-1} (n-i)^{\ell_i}.
  \end{equation*}
  The assertion follows.
\end{proof}

\begin{remark}\label{remark:goal-klequ}
  Let $k_0$, \ldots, $k_{n-1}$ and $\ell_0$, \ldots,
  $\ell_{n-1} \in \N_0$ be the exponents of the factors $f$ and $g$ of
  $\binom{x}{n}^m$, cf.~Proposition~\ref{proposition:intfactors}.  If
  $k_0 = k_1 = \cdots = k_{n-1}$ (and hence
  $\ell_0 = \ell_1 = \cdots = \ell_{n-1}$), then
  $f = \binom{x}{n}^{k_0}$ and $g = \binom{x}{n}^{\ell_0}$.

  Having Remark~\ref{remark:rephrase-goal} in mind, we aim to show
  that, for all possible factors $f$ and $g$, the corresponding
  exponents satisfy $k_0 = k_1 = \cdots = k_{n-1}$ and
  $\ell_0 = \ell_1 = \cdots = \ell_{n-1}$ in order to prove that the
  binomial polynomial $\binom{x}{n}$ is absolutely irreducible.
\end{remark}

We reformulate the task at hand into a homogeneous system of linear
equations for which the vectors $(k_0,\ldots, k_{n-1})^t$ and
$(\ell_0,\ldots, \ell_{n-1})^t$ are solutions.  Given the form of
possible factors $f$ and $g$ of $\binom{x}{n}^m$ in $\Int(\Z)$ (see
Proposition~\ref{proposition:intfactors}), it is natural to rephrase
the integer-valued condition in terms of $p$-adic valuations.

\begin{notation}\label{notation:valuation}
  For $w \in \mathbb{Q}$ and a prime number $p$, we denote by
  $\val_p(w)$ the $p$-adic valuation of $w$.
\end{notation}

\begin{remark}\label{remark:ivall}
  Let $k_0$, \ldots, $k_{n-1}$ and $\ell_0$, \ldots, $\ell_{n-1}$ be
  the exponents of the factors $f$ and $g$ of $\binom{x}{n}^m$,
  cf.~Proposition~\ref{proposition:intfactors}.  Since $f$ and $g$ are
  integer-valued polynomials, it follows that
  \begin{align*}
    \val_p(f(s)) &= \sum_{j=0}^{n-1} (\val_p(s-j) - \val_p(n-j))k_j  \geq 0 \text{ and }\\
    \val_p(g(s)) &= \sum_{j=0}^{n-1} (\val_p(s-j) - \val_p(n-j))\ell_j \geq 0
  \end{align*}
  for all $s \in \Z$ and all prime numbers $p$.
\end{remark}

For our purposes, it turns out to be sufficient to consider $p$-adic
valuations for prime numbers $p\le n$ and integers of the form
$s = n+r$ for $r \in \{1,\ldots, p-\rempn-1\}$ where
$\rempn$ is the uniquely determined integer with $0 \le \rempn \le p-1$ and
$n \equiv \rempn \mod p$. However, there are two cases for which this
range of $r$ is not sufficient, that is, when $n=2^s$ with $s>1$ and
$p=2$ we also need $r=2$ and in case $n=9$ and $p=3$ we also need $r=3$ and $r=4$.  This motivates the following notation which we
use throughout the remainder of this paper.

\begin{notation}\label{notation:pr}
  For $n\in \N$, we use the following notation. 
  \begin{enumerate}
  \item $\calP_n = \{p \mid 0 < p\le n \text{ prime number}\}$
  \item For $p\in \calP_n$, let $0\le \rempn < p$
    be the uniquely determined integer with $n \equiv \rempn \mod p$.
  \item For $p\in \calP_n$, we set
    \begin{equation*}
      \calR_{n,p} = 
      \begin{cases}
        \{1,2\} & \text{ if } n = 2^s \text{ with } s>1 \text{ and } p=2 \\
        \{1,2,3,4\} & \text{ if } n=9 \text{ and } p=3 \\
        \{r \mid 1\le r \le p -\rempn-1\} & \text{ else. }
      \end{cases}
    \end{equation*}
  \end{enumerate}
\end{notation}

As mentioned above, our goal is to reformulate the question whether
the binomial polynomial $\binom{x}{n}$ is absolutely irreducible as a
homogeneous linear equation system. 

\begin{definition}\label{definition:A}
  For $n\ge 2$, we define the \emph{valuation matrix} $\A_n$ of $n$ by
  \begin{equation*}
    \A_n = \left(\val_p(n+r-j)-\val_p(n-j) \right)_{\substack{p\in \calP_n, r\in \calR_{n,p} \\ 0 \le j \le n-1}},
  \end{equation*}
  cf.~Notation~\ref{notation:pr}. 
\end{definition}

\begin{remark}\label{remark:Ak>=0}
  Let $\mathbf{k} = (k_0, \ldots, k_{n-1})^t$ and
  $\mathbf{l}=(\ell_0, \ldots, \ell_{n-1})^t \in \N_0^n$ be (the
  vectors of) the exponents of the factors $f$ and $g$ of
  $\binom{x}{n}^m$, cf.~Proposition~\ref{proposition:intfactors}.  The
  condition that $f$ and $g$ are integer-valued immediately implies
  that
  \begin{equation*}
    \A_n\mathbf{k} \ge 0 \text{ and } \A_n\mathbf{l} \ge 0,
  \end{equation*}
  cf.~Remark~\ref{remark:ivall}.
\end{remark}

Our next step is to show that the exponent vectors
$\mathbf{k} = (k_0, \ldots, k_{n-1})^t$ and
$\mathbf{l}=(\ell_0, \ldots, \ell_{n-1})^t\in \N_0^n$ of possible integer-valued
factors $f$ and $g$ of $\binom{x}{n}^m$ are actually solutions to the
homogeneous linear equation system $\A_nx = 0$,
cf.~Remark~\ref{remark:Ak>=0}. Before we prove this in
Proposition~\ref{proposition:kernelA}, we need the following lemma
which states that the row sums of $\A_n$ equal zero.

\begin{lemma}\label{lemma:equality}
  Let $n\in \N$, $p\in \calP_n$ and $r\in \calR_{n,p}$,
  cf.~Notation~\ref{notation:pr}.
  Then
  \begin{equation*}
   \sum_{j=0}^{n-1} (\val_p(n+r-j)-\val_p(n-j)) = 0.
  \end{equation*}
\end{lemma}
\begin{proof}
  Let $p\in \calP_n$ and $q$ and $\rempn$ be the uniquely determined
  integers such that $n = qp+\rempn$ and $0\le \rempn \le p-1$. First,
  we assume that $1\le r \le p-\rempn-1$ (this covers all cases except
  the one where $n$ is a power of 2 and $r = p = 2$ and the one where
  $r\in \{3,4\}$, $p=3$ and $n=9$).  Then for all $0 \le j \le n-1$ at
  most one of the numbers $n+r-j$ and $n-j$ is divisible by $p$.
  Since
  \begin{equation*}
    2 = n +1 - (n-1) \le n+r-j\le n+(p-\rempn-1) = qp+(p-1), 
  \end{equation*}
  it follows that $n+r-j = kp$ if and only if $k\in \{1,\ldots,q\}$
  and hence
  \begin{equation*}
    \sum_{j=0}^{n-1} \val_p(n+r-j) = \sum_{k=1}^{q}\val_p(kp).
  \end{equation*}
  Similarly, since $1 \le n-j \le n = qp+\rempn$, it follows that
  \begin{equation*}
    \sum_{j=0}^{n-1} \val_p(n-j) = \sum_{k=1}^{q}\val_p(kp)
  \end{equation*}
  which proves the assertion in this case.
  
  Next, we discuss the case where $n=2^s$ is a power of 2 with $s > 1$ and
  $r = p=2$. We want to show that
  $\sum_{j=0}^{n-1} (\val_2(n+2-j)-\val_2(n-j)) = 0$.

  Let $v_j = \val_2(n+2-j)-\val_2(n-j)$. As $n$ is even,
  $v_j=0$ whenever $j$ is odd. Moreover, for $j\in 4\Z$,
  $\val_2(n+2-j) = \val_2(n-2-j) = 1$ and $\val_2(n-j) > 1$.
  Therefore, for $j\in 4\Z$, 
  \begin{equation*}
    v_j + v_{j+2} = \val_2(n+2-j)-\val_2(n-j) + \val_2(n-j)-\val_2(n-2-j) = 0
  \end{equation*}
  which implies that
  \begin{align*}
    \sum_{j=0}^{n-1} v_j = \sum_{\substack{j=0 \\ j\in 2\Z}}^{n-1} v_j = \sum_{\substack{j=0 \\ j\in 4\Z}}^{n-1} v_j + v_{j+2} = 0.
  \end{align*}

  Finally, it remains to discuss the case where $n=9$, $p=3$ and
  $r=3$, $4$. The two rows are
  \begin{equation*}
    \begin{pmatrix}
      -1 & 0 & 0 &  1 & 0 & 0 & 0 & 0 & 0 \\
      -2 & 1 & 0 & -1 & 2 & 0 & -1 & 1 & 0 \\
    \end{pmatrix}.
  \end{equation*}
  An easy computation verifies the claim and the assertion follows.
  \end{proof}

  Lemma~\ref{lemma:equality} is the key to show that the exponent
  vectors $\mathbf{k} = (k_0, \ldots, k_{n-1})^t$ and
  $\mathbf{l}=(\ell_0, \ldots, \ell_{n-1})^t$ of potential factors $f$
  and $g$ of $\binom{x}{n}^m$ can be viewed as the solutions of the
  homogeneous equation system $\A_n x = 0$, or equivalently, as
  elements of $\ker(\A_n)$.

\begin{proposition}\label{proposition:eqs}
  Let $n$, $m\ge 2$ be integers and $\mathbf{k} = (k_0, \ldots, k_{n-1})^t$ and
  $\mathbf{l} = (\ell_0, \ldots, \ell_{n-1})^t \in \N_0^n$ such that
  \begin{equation*}
    \prod_{i=0}^{n-1} \left( \frac{x-i}{n-i}\right)^{k_i}, \prod_{i=0}^{n-1} \left( \frac{x-i}{n-i}\right)^{\ell_i} \in \Int(\Z)
  \end{equation*}
  and $k_j + \ell_j = m$ for all $0\le j \le n-1$. 

  Then $\mathbf{k}$, $\mathbf{l} \in \ker(\A_n)$. 
\end{proposition}

\begin{proof}
  The integer-valued condition in the hypothesis implies that
  $\A_n\mathbf{k} \ge 0$ and $\A_n\mathbf{l} \ge 0$,
  cf.~Remark~\ref{remark:Ak>=0}.

  Moreover, by Lemma~\ref{lemma:equality}, 
  \begin{align*}
    0 &= m \cdot \left(\sum_{j=0}^{n-1} (\val_p(n+r-j)-\val_p(n-j))\right)_{p\in \calP_n, r\in \calR_{n,p}} \\
      &= \left(\sum_{j=0}^{n-1} (\val_p(n+r-j)-\val_p(n-j))\cdot (k_j + l_j)\right)_{p\in \calP_n, r\in \calR_{n,p}} \\
      &= \A_n(\mathbf{k}+\mathbf{l}) = \A_n \mathbf{k} + \A_n \mathbf{l} \\
  \end{align*}
  holds and therefore $\A_n \mathbf{k}= \A_n \mathbf{l} = 0$.
\end{proof}

\begin{remark}\label{remark:dimker>0}
  It follows from Lemma~\ref{lemma:equality} that
  $(k)_{0\le j \le n-1}$ are elements of $\ker(\A_n)$ for all
  $k\in \Q$. In particular, this implies that
  $(1)_{0\le j \le n-1} \in \ker(\A_n)$ and 
  $\rank(\A_n) < n$.

  Moreover, note that the exponent vector of $\binom{x}{n}^k$ (written
  in the form of Proposition~\ref{proposition:intfactors}) is a scalar
  multiple of the element $(1)_{0\le j \le n-1}$.
\end{remark}

This brings us to the main result of this section which states that,
in order to show that $\binom{x}{n}$ is absolutely irreducible, it
suffices to prove that the rank of the valuation matrix $\A_n$ of $n$
is exactly $n-1$.

\begin{proposition}\label{proposition:kernelA}
  Let $n\ge 2$ be an integer and $\A_n$ the valuation matrix of~$n$.

  If $\rank(\A_n) = n-1$, then $\binom{x}{n}$ is absolutely irreducible.
\end{proposition}

\begin{proof}
  Let $m\ge 2$ be an integer and assume that $f$, $g\in \Int(\Z)$ such
  that $\binom{x}{n}^m = f\cdot g$. According to
  Propositions~\ref{proposition:intfactors} and~\ref{proposition:eqs}
  it follows that
  \begin{equation*}
    f = \pm \prod_{i=0}^{n-1} \left(\frac{x-i}{n-i} \right)^{k_i} \text{ and }
    g = \pm \prod_{i=0}^{n-1} \left(\frac{x-i}{n-i}\right)^{\ell_i}
  \end{equation*}
  where $(k_0,\ldots, k_{n-1})^t$,
  $(\ell_0,\ldots, \ell_{n-1})^t \in \ker(\A_n) \cap \N_0^{n}$. 

  Since $\dim\ker(\A_n) = n-\rank(\A_n)= 1$ and
  $(1)_{0\le j \le n-1} \in \ker(\A_n)$ by
  Lemma~\ref{lemma:equality}, it follows that
  $\ker(\A_n) = \spn_\Q\{(1)_{0\le j\le n-1}\}$. Therefore,
  $k_0 = k_1 = \cdots = k_{n-1}$ and
  $\ell_0 = \ell_1 = \cdots = \ell_{n-1}$ which implies that
  $f = \binom{x}{n}^{k_0}$ and $g = \binom{x}{n}^{\ell_0}$.

  It follows that $\binom{x}{n}$ is absolutely irreducible
  by~Remark~\ref{remark:rephrase-goal}.
\end{proof}

To be able to prove that $\rank(\A_n) = n-1$ for all integers $n\ge 2$
in Section~\ref{section:rank}, we first need to show
Theorem~\ref{theorem:numbertheory} which is part of our
number-theoretic toolbox.

\section{Number-theoretic toolbox}
\label{section:numbertheory}

The goal of this section is to prove
Theorem~\ref{theorem:numbertheory}. Given an integer $n>10$, let $P$
be the maximal prime number with $P\le n$. We show that for every
$2\le k\le n-P$ there exists a prime number $p>2k$ which divides one of the
numbers $n$, $n-1$, \ldots, $n-k+1$. Note that the condition
$k\le n-P$ is equivalent to $n$, $n-1$, \ldots, $n-k+1$ being
composite numbers.

The literature provides us a collection of number-theoretic facts
which, putting the pieces together, give a proof of
Theorem~\ref{theorem:numbertheory}. We split the proof into cases and
present partial results on their own. We start with the case for large
$n$.

\begin{proposition}\label{proposition:nthlargen}
  Let $2 \leq k<n$ be positive integers with $n\ge 4,021,520$.

  If $n$, $n-1$,\ldots, $n-k+1$ are composite numbers, then
  one of them has a prime divisor $p>2k$. 
\end{proposition}

For the proof we use the following facts from the literature.

\begin{fact}[{Bertrand's postulate}]\label{fact:bertrand}
  For all integers $n\ge 3$, there exists a prime number $p$ with
  $\frac{n}{2} < p < n$.
\end{fact}

\begin{fact}[{\cite[Theorem~12]{Schoenfeld:1976:chebyshev}}]\label{fact:nextprimebound}
  Let $m$ be an integer with $m \geq 2,010,760$. Then there exists a
  prime number $p$ such that $m < p < (1 + \frac{1}{16597})m$.
\end{fact}

\begin{fact}[{\cite[Theorem~1]{Laishram-Shorey:2005:greatestprimedivisor}}]\label{fact:mkbound}
  Let $m$ and $k \geq 2$ be positive integers such that
  $m > \max \{k + 13, \frac{279}{262}k \}$.

  Then the product $m(m+1) \cdots (m+k-1)$ has a prime factor greater
  than~$2k$.
\end{fact}

\begin{proof}[Proof of Proposition~\ref{proposition:nthlargen}]
  Let $P$ be the largest prime number with $P\le n$. First, we prove
  the following

  \textbf{Claim.} $P + 1 > \max\{(n-P)+13, \frac{279}{262}(n-P)\}$

  Assume for a moment that the claim holds. By hypothesis, the numbers
  $n$, $n-1$, \ldots, $n-k+1$ are composite numbers which implies that
  $n-k+1 > P$ or, equivalently, $n-P \ge k$. Therefore
  \begin{equation*}
    n-k+1 \ge P+1 >\max\left\{(n-P)+13, \frac{279}{262}(n-P)\right\}\ge \max\left\{k+13, \frac{279}{262}k\right\}
  \end{equation*}
  and the assertion follows from Fact~\ref{fact:mkbound} with
  $m=n-k+1$.
  
  It remains to prove the claim.  Let $Q$ be the smallest prime number
  with $Q>n$. As $P>\frac{n}{2}\ge 2,010,760$ holds according to
  Fact~\ref{fact:bertrand}, it follows from
  Fact~\ref{fact:nextprimebound} that $Q < (1+ \frac{1}{16597})P$.
  
  Therefore, $2P - n > 2P-Q > (1-\frac{1}{16597})P > 12$ which implies
  that $P+1 > (n-P)+13$. Moreover, since
  $2(1+\frac{1}{16597}) < 1+\frac{279}{262}$ it follows that
  $\frac{279}{262}n -1 < 2n < 2Q < (1+\frac{279}{262})P $ and hence
  $P+1 > \frac{279}{262}(n-P)$. This completes the proof of the claim.
\end{proof}

It remains to discuss the case where $n < 4,021,520$. Here, we can use
a result of Laishram and
Shorey~\cite{Laishram-Shorey:2006:grimm}. They proved Grimm's
conjecture for sequences of consecutive positive integers whose
minimum does not exceed a certain bound.
\begin{fact}[{\cite[Theorem~1]{Laishram-Shorey:2006:grimm}}]\label{fact:grimmsmalln}
  Let $k<n$ such that $n-k+1 \le 1.9\cdot 10^{10}$.

  If $n$, $n-1$, \ldots, $n-k+1$ are composite numbers then there
  exist pairwise distinct prime numbers $p_0$, $p_1$, \ldots, $p_{k-1}$ with
  $p_i \mid n-i$ for all $0\le i \le k-1$.
\end{fact}

\begin{remark}\label{remark:needsmalln}
  Let $p_k$ denote the $k$-th prime number. It is easily seen that
  $p_k > 2k$ for $k\ge 5$. Therefore, under the hypothesis of
  Fact~\ref{fact:grimmsmalln}, there exists a prime number $p>2k$
  which divides one of the numbers $n$, $n-1$, \ldots, $n-k+1$.
\end{remark}

We treat the cases $k=2$, $3$, $4$ in the proof below where we use the
following facts.
\begin{fact}[{Catalan's conjecture, \cite{Mihailescu:2004:catalan}}]\label{fact:catalan}
  Mih\u{a}ilescu proved Catalan's conjecture which states that the only
  consecutive positive integers which are non-prime prime powers are $8$
  and $9$.
\end{fact}

\begin{fact}[{Pillai's conjecture for difference 2}, \cite{oeis:pillai}]\label{fact:pillai}
  The only non-prime prime powers $p^x$ and $q^y$ less than $10^{18}$
  with $p^x-q^y=2$ are $25$ and $27$.
\end{fact}

Finally, we restate the desired theorem and give a complete proof. 
\begin{reptheorem}{theorem:numbertheory}
  Let $n>10$ be an integer and $P$ the maximal prime number with $P\le n$. 

  If $2 \le k\le n-P$ then there exists a prime number $p>2k$ which divides
  one of the numbers $n$, $n-1$, \ldots, $n-(k-1)$.
\end{reptheorem}

\begin{remark}\label{remark:dealwithpowersof2}
  It is immediately clear that $n$ has a prime divisor $p > 2$ ($k=1$)
  if and only if $n$ is not a power of $2$.

  It turns out that we need to work around the lack of other prime
  divisors in case $n=2^x$, cf.~Definition~\ref{definition:A}.
\end{remark}

\begin{proof}
  For $n\ge 4,021,520$, the assertion follows from
  Proposition~\ref{proposition:nthlargen}.

  For $n <4,021,520$ and $k\ge 5$, the assertion follows from
  Fact~\ref{fact:grimmsmalln}, cf.~Remark~\ref{remark:needsmalln}.

  We treat the cases $k=2$, $3$, $4$ separately.

  \textbf{Case $k=2$:} Note that $n$ and $n-1$ are composite numbers
  which have no common prime factors. Moreover, not both of them are
  prime powers by Fact~\ref{fact:catalan} which implies that $n(n-1)$
  has three pairwise distinct prime factors, one of which is
  necessarily at least $5 > 2\cdot 2$.
  
  \textbf{Case $k=3$:} We use a similar argument as above. If $n = 27$
  then $p = 13 > 2\cdot 3$ divides $n-1 = 26$. By
  Fact~\ref{fact:pillai}, we can therefore assume that at most one of
  the numbers $n$, $n-1$ and $n-2$ is a prime power and the other two
  are composite numbers with at least two distinct prime divisors
  each. Since $n$, $n-1$, and $n-2$ can have at most the prime factor
  $2$ in common, it follows that their product $n(n-1)(n-2)$ has at
  least four pairwise distinct prime divisors. In particular, one of
  these prime divisors is at least $7>2\cdot 3$.

  \textbf{Case $k=4$:} If $27 \in \{n,n-1\}$, then $p = 13> 2\cdot 4$
  divides $n-1$ or $n-2$.  We can therefore assume that either at most
  one of the numbers $n$, $n-1$, $n-2$, $n-3$ is a prime power or $n$
  and $n-3$ are the only prime powers among them
  (cf.~Fact~\ref{fact:pillai}).  First, we assume that at most one of
  the numbers $n$, $n-1$, $n-2$ and $n-3$ is a prime power and the
  other three are composite numbers with at least two distinct prime
  divisors each.  Since $n$, $n-1$, $n-2$ and $n-3$ can have at most
  the prime factors $2$ (occurring twice) and $3$ (occurring at most
  twice) in common, it follows that their product $n(n-1)(n-2)(n-3)$
  has at least five pairwise distinct prime divisors. In particular,
  one of these prime divisors is at least $11>2\cdot 4$. Finally, we
  assume that both $n$ and $n-3$ are prime powers and $n-1$ and $n-2$
  are composite numbers with at least two distinct prime divisors
  each. Note that $n$ and $n-3$ are prime powers of two distinct
  primes and both are not divisible by $3$. Therefore, $n$, $n-1$,
  $n-2$ and $n-3$ can have at most the prime factor $2$ (occurring
  twice) in common. Again, it follows that their product
  $n(n-1)(n-2)(n-3)$ has at least five pairwise distinct prime
  divisors, one of them being at least $11$.
\end{proof}


\section{\texorpdfstring{$\rank(\A_n)$}{rank(An)} and the \texorpdfstring{$p$}{p}-blocks}
\label{section:rank}
In this section we show that $\rank(\A_n) = n-1$ for all $n\ge
2$. Once this is shown, it follows from
Proposition~\ref{proposition:kernelA} that $\binom{x}{n}$ is
absolutely irreducible.

Our strategy is to show that the columns
$j\in \{0,1,\ldots, n-1\} \setminus\{n-P\}$ are linearly independent
where $P = \max\calP_n$ is the maximal prime number which is less than
or equal to $n$. To do so, we split these $n-1$ columns of $\A_n$ into
two groups, namely the ``outer'' $2(n-P)$ columns indexed with
$\{0,\ldots, n-P-1\}\cup \{P,P+1,\ldots, n-1\}$ and the ``inner''
$2P-n-1$ columns indexed with $\{n-P+1, \ldots, P-1\}$.  According to
the next lemma, $n-P< P-1$ which implies that we can always partition
the $n-1$ columns in this way.

\begin{lemma}\label{lemma:column-grouping}
  Let $n\ge 2$ be an integer and $P = \max\calP_n$ be the largest prime
  number which is less than or equal to $n$
  (cf.~Notation~\ref{notation:pr}).

  Then $P-1 > n-P$.
\end{lemma}
\begin{proof}
  According to Bertrand's postulate (see Fact~\ref{fact:bertrand}),
  there is a prime number $q$ with $P < q < 2P$. Since $P=\max\calP_n$
  is the maximal prime number at most $n$, it follows that
  $n < q < 2P$. Therefore, $n < 2P-1$ which is equivalent to
  $n-P < P-1$.
\end{proof}

Next, we introduce the $\Q$-spans of groups of columns of a matrix that
are of interest for our investigation. This is done slightly more
general as the two column groups of $\A_n$ we mentioned above. On the
one hand, we want to switch between the whole matrix $\A_n$ and
certain submatrices which we introduce below (the $p$-blocks).  On the
other hand, to show that the outermost $2(n-P)$ columns of $\A_n$ are
linearly independent, we use an inductive argument for which the
definition below is convenient.

\begin{definition}
  Let $n\ge 2$ and $\ell$ be positive integers and
  $C \in \Q^{\ell\times n}$ with columns $c_0$, \ldots,
  $c_{n-1}$. Further, let $P=\max\calP_n$ be the largest prime number
  less than or equal to $n$, cf. Notation~\ref{notation:pr}.
  \begin{enumerate}
  \item We denote by
    $\inspn(C) = \spn\{c_{n-P+1}, c_{n-P+2},\ldots, c_{P-1}\}$ be the
    span of the ``inner'' $2P-n-1$ columns of $C$.
  \item For $0 \le s < t \le n-P-1$, we denote by
    \begin{equation*}
      \outspn{s}{t}(C) = \spn\{c_j \mid s\le j \le t \text{ and } n-(t+1)\le j \le n-(s+1)\}
    \end{equation*}
    be the $\Q$-span of the $2(t-s+1)$ columns in the range between
    $s$ and $t$ in the left half of $C$ and in the range between $n-(t+1)$ and
    $n-(s+1)$ in the right half of $C$ (the same range of columns on the
    ``left'' and the ``right'' side of $C$).
  \end{enumerate}
  For a visualization of the generating columns of these spans, see
  Figure~\ref{fig:v-span}.
\end{definition}

\begin{figure}[h]
  \centering
  \begin{tikzpicture}[mymatrixenv]
  \matrix[mymatrix] (m)  {
 |[alias=b1o]|{}& {}             &{}&|[alias=mvb1o]|{}  &|[alias=vb1o]|{ }&{}               &|[alias=mclo]|{} &{}&{}               &|[alias=mo]|{}&|[alias=inspn]|{} &{}              &{}              &|[alias=mclro]|{}& |[alias=1b3o]|{ }             & {}             &{}& {}                & {}              & {}              &{}\\%
 {}             &  \\
 {}&|[scale=1.5]|{\hspace*{-0.2em}*}& &                   &                 &                 &                 &  &|[scale=1.5]|{\hspace*{0.55em} *}  & \hspace*{1.6em}&                &|[scale=1.5]|{\hspace*{3.1em} *\hspace*{-0.1em}}&                  &                                  & |[scale=1.5 ]|{} & &                & &   |[scale=1.5 ]|{\hspace*{0.5em}*\hspace*{-0.9em}} & \\
 {}             & \\                   
 |[alias=b1u]|{}&                &  &|[alias=mvb1u]|{  }&                 &                 & |[alias=vb1u]|{}&  &                &|[alias=mu]|{} &|[alias=inspnu]|{} &|[alias=jPu]|{} &                &|[alias=mclru]|{} &|[alias=np1u]|{}                  &                   & &|[alias=1b3u]|{}&                  &   &|[alias=b1ur]|{}  \\
 \\
   };
   %

    \mymatrixbracetop{1}{3}{\footnotesize $s$}
    \mymatrixbracetop{4}{7}{\footnotesize $t-s+1$}
    \mymatrixbracetop{8}{10}{\footnotesize $n-P-t-1$}
    \mymatrixbracetop{12}{14}{\footnotesize $n-P-t-1$}
    \mymatrixbracetop{15}{18}{\footnotesize $t-s+1$}
    \mymatrixbracetop{19}{21}{\footnotesize $s$}




    \node[below= of mvb1u] (j) {\footnotesize $j=s$};
    \draw[->, shorten >= 0.4cm] (j.north) -- (mvb1u.south);

    \node[below= of vb1u] (j2) {\footnotesize $j=t$};
    \draw[->, shorten >= 0.4cm] (j2.north) -- (vb1u.south);

    \node[below= of 1b3u] (rp1) {\footnotesize $j = n-(s+1)$};
    \draw[->, shorten >= 0.4cm] (rp1.north) -- (1b3u.south);

    \node[below= 1.5cm of np1u] (np1) {\footnotesize $j = n-(t+1)$};
    \draw[->, shorten >= 0.4cm] (np1.north) -- (np1u.south);

    \coordinate (ign) at ($(mo)+(+0.8em,0.45em)$);
    \coordinate (ignu) at ($(mu)+(+0.8em,-0.5em)$);
     \draw[red, thick, dashed] (ign) -- (ignu);
     \node[above= 1cm of ign] (nP) {\footnotesize $j = n-P$};
    \draw[->, shorten >= 0.4cm] (nP.south) -- (ign.north);

    \node[above right= 0.5cm and -0cm of inspn] (in) {\footnotesize $\inspn(C)$};

    \coordinate (Wlo) at ($(mo)+(+1.5em,0.7em)$);
    \coordinate (Wro) at ($(mo)+(+5em,0.7em)$);
    \coordinate (Wlu) at ($(mu)+(+1.5em, -0.7em)$);
    \coordinate (Wru) at ($(mu)+(+5em, -0.7em)$);
     \draw[fill=Gray, opacity=0.2] (Wlo)
    [snake=zigzag, segment length=10pt]-- (Wlu)
    [snake=zigzag, segment length=100pt]-- (Wru)
    [snake=zigzag, segment length=10pt]--(Wro)
    [snake=zigzag, segment length=100pt]--(Wlo);
    \draw[Gray, thick] ($(ign.north)+(1.2em,0)$) -- ($(ignu.south)+(1.2em,0)$);
    \draw[Gray, thick] ($(ign.north)+(1.7em,0)$) -- ($(ignu.south)+(1.7em,0)$);
    \draw[Gray, thick] ($(ign.north)+(2.2em,0)$) -- ($(ignu.south)+(2.2em,0)$);
    \draw[Gray, thick, dotted] ($(ign.north)+(2.6em,-0.75cm)$) -- ($(ign.north)+(3.2em, -0.75cm)$);
    \draw[Gray, thick] ($(ign.north)+(3.7em,0)$) -- ($(ignu.south)+(3.7em,0)$);

    \node[below right= 1.3cm  and -0.95cm of Wlu] (nP1) {\footnotesize $j = n-P+1$};
    \draw[->, shorten >= 0.15cm] (nP1.north) -- ($(Wlu)+(0.5em,0)$);

    \node[below right= 0.75cm  and -1.02cm of Wru] (n1) {\footnotesize $j = P-1$};
    \draw[->, shorten >= 0.15cm] (n1.north) -- ($(Wru)+(-0.5em,0)$);

    \draw[ForestGreen, thick] ($(mvb1o.north)+(0,0.4em)$) -- ($(mvb1u.south)+(0,-0.4em)$);
    \draw[ForestGreen, thick] ($(mvb1o.north)+(0.5em,0.4em)$) -- ($(mvb1u.south)+(0.5em,-0.4em)$);
    \draw[ForestGreen, thick] ($(mvb1o.north)+(1em,0.4em)$) -- ($(mvb1u.south)+(1em,-0.4em)$);
    \draw[ForestGreen, thick, dotted] ($(mvb1o.north)+(1.5em,-0.6cm)$) -- ($(mvb1o.north)+(3em, -0.6cm)$);
    \draw[ForestGreen, thick] ($(mclo.north)+(0,0.4em)$) -- ($(vb1u.south)+(0, -0.4em)$);

    \draw[ForestGreen, thick] ($(1b3o.north)+(0,0.4em)$) -- ($(np1u.south)+(0,-0.4em)$);
    \draw[ForestGreen, thick] ($(1b3o.north)+(0.5em,0.4em)$) -- ($(np1u.south)+(0.5em,-0.4em)$);
    \draw[ForestGreen, thick] ($(1b3o.north)+(1em,0.4em)$) -- ($(np1u.south)+(1em,-0.4em)$);

    \draw[ForestGreen, thick, dotted] ($(1b3o.north)+(1.5em,-0.6cm)$) -- ($(1b3o.north)+(3em, -0.6cm)$);
    \draw[ForestGreen, thick] ($(1b3o.north)+(3.35em,0.4em)$) -- ($(np1u.south)+(3.35em, -0.4em)$);

    \draw[ForestGreen, rounded corners] ($(mvb1o.north)+(-0.4em,0.8em)$) rectangle ($(vb1u)+(0.4em,-0.7em)$);
    \draw[ForestGreen, rounded corners] ( $(1b3o.north)+(-0.4em,0.8em)$) rectangle ( $(1b3u)+(0.4em,-0.7em)$);

\end{tikzpicture}

  \caption{$\inspn(C)$ is spanned by the columns in the (grey)
    zigzag-area; $\outspn{s}{t}(C)$ is spanned by the columns in the two (green)
    rectangles on the left and right.}
  \label{fig:v-span}
\end{figure}
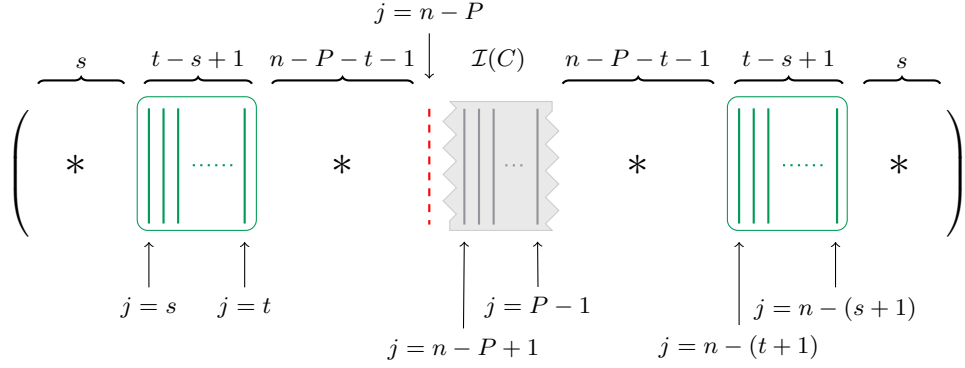

\begin{remark}\label{remark:directsum}
  Let $n\in \N$ and $P = \max\calP_n$,
  cf.~Notation~\ref{notation:pr}. 

  Our goal is to show that
  \begin{enumerate}[(i)]
  \item $\dim\outspn{0}{n-P-1}(\A_n) = 2(n-P)$, 
  \item $\dim\inspn(\A_n) = 2P-n-1$ and
  \item $\inspn(\A_n) \cap \outspn{0}{n-P-1}(\A_n) = \mathbf{0}$.
  \end{enumerate}
  Together this then implies that the sum
  $\inspn(\A_n) + \outspn{0}{n-P-1}(\A_n)$ is direct and
  \begin{equation*}
    \dim\left(\inspn(\A_n) \oplus \outspn{0}{n-P-1}(\A_n)\right) = 2(n-P) + (2P-n-1) = n-1.
  \end{equation*}
  Since $n > \rank(\A_n)$ by Remark~\ref{remark:dimker>0}, it further follows that 
  \begin{equation*}
    n > \rank(\A_n) \ge \dim\left(\inspn(\A_n) \oplus \outspn{0}{n-P-1}(\A_n)\right) = n-1
  \end{equation*}
  and hence
  \begin{equation*}
    \rank(\A_n) = n-1.
  \end{equation*}
  Then, by Proposition~\ref{proposition:kernelA}, $\binom{x}{n}$ is
  absolutely irreducible.

  Note that Assertion~(i) is shown in
  Proposition~\ref{proposition:outerdim} and Assertions~(ii) and~(iii)
  are proven in Corollary~\ref{corollary:innerdim} below.
\end{remark}

For our further investigation, it turns out to be useful to split the
valuation matrix $\A_n$ into blocks of rows, each block corresponding
to a prime number $p\in \calP_n$.

\subsection{The structure of a \texorpdfstring{$p$}{p}-block}
\label{subsection:p-block-structure}

\begin{definition}\label{definition:pblock}
  For $n\in \N$ and $p\in\calP_n$, we define the \emph{$p$-block}
  $\pblock{n}{p}$ as the $|\calR_{n,p}| \times n$ integer matrix
  defined by
  \begin{equation*}
    \pblock{n}{p} = \left(\val_{p}(n+r-j)-\val_p(n-j)\right)_{\substack{r\in \calR_{n,p}\\ 0\le j\le n-1}}
  \end{equation*}
\end{definition}

For our purposes, we focus on the distribution of zero and non-zero
entries of leftmost $p$ and rightmost $p-1$ columns of the
$p$-blocks. In this sense $\pblock{n}{p}$ has a ``structure'' which is
described in Propositions~\ref{proposition:lefthalfpblock}
and~\ref{proposition:righthalfpblock}.

\begin{remark}\label{remark:exceptions}
  Note that in the exceptional cases where
  $\calR_{n.p} \neq \{r \mid 1\le r \le p-\rempn-1\}$,
  Propositions~\ref{proposition:lefthalfpblock}
  and~\ref{proposition:righthalfpblock} only give information about
  the first $p-\rempn-1$ rows of the corresponding
  $p$-blocks. However, in all situations where we apply the two
  propositions below, we can rule out these exceptional cases. 
\end{remark}

\begin{proposition}\label{proposition:lefthalfpblock}
  Let $n\in\N$, $p \in \calP_n$ (cf.~Notation~\ref{notation:pr}) and
  $\pblock{n}{p} = (b_{r,j})$ the corresponding $p$-block.

  For $0 \le j \le p-1$ and $1 \le r\le p-\rempn-1$ the
  following holds:
  \begin{equation*}
    b_{r,j} =
    \begin{cases}
      -\val_p(n-\rempn) & j = \rempn \\
       \val_p(n-\rempn) & j = r + \rempn \\
       0 & \text{else}.
    \end{cases}
  \end{equation*}
  where, as usual, $0 \le \rempn < p$ with $n \equiv \rempn \mod p$.
\end{proposition}

A visualization of the statement in
Proposition~\ref{proposition:lefthalfpblock} can be found in
Figure~\ref{fig:pblock-left}.

\begin{figure}[h]
  \centering
    \begin{tikzpicture}[mymatrixenv]
  \matrix[mymatrix] (m)  {
 |[alias=b1o]|{}& {}             &{}&|[alias=mvb1o]|{-v}&|[alias=vb1o]|{v}&{}               & {}&|[alias=mclo]|{} &{}&{}               &  & \\
 {}             &  \\                                                                                         
 {}             &|[scale=1.5]|{0}&  &                   &                 & |[scale=1.5 ]|{}& {}&                 &  &|[scale=1.5]|{*} &  & \\
 {}             & \\                                                                                                             
 |[alias=b1u]|{}&                &  &|[alias=mvb1u]|{-v}&                 &                 & {}& |[alias=vb1u]|v &  &                 &  & \\
 \\                                                                                         
   };
   %
    \mymatrixbraceleft{1}{5}{\footnotesize $p-\rempn-1$}

    \mymatrixbracetop{1}{3}{\footnotesize $\rempn$}
    \node[above= of mvb1o] (j) {\footnotesize $j=\rempn$};
    \draw[->, shorten >= 0.3cm] (j.south) -- (mvb1o.north);

    \node[above=0.8cm of mclo] (p1) {\footnotesize $j = p-1$};
    \draw[->, shorten >= 0.5cm] (p1.south) -- (mclo.north);







    \draw[dotted,shorten <=0.1cm,shorten >=0.1cm] (mvb1o.south) -- (mvb1u.north);
    \draw[dotted,shorten <=0.1cm,shorten >=0.1cm] (vb1o.south east) -- (vb1u.north west);


    \draw[ForestGreen, rounded corners] ($(b1o.north)+(-0.4em,0.8em)$) rectangle ($(vb1u)+(0.4em,-0.7em)$);

\end{tikzpicture}

    \caption{Leftmost $p$ columns of (the first $p-\rempn-1$ rows of) $\pblock{n}{p}$ where $v = \val_p(n-\rempn)$,
      cf.~Proposition~\ref{proposition:lefthalfpblock}}
  \label{fig:pblock-left}
\end{figure}

\begin{proof}
  Let $0\le j \le p-1$.  In this range of $j$ it follows from the
  definition of $\rempn$ that
  \begin{equation*}
    \val_p(n-j) \neq 0 \;\Longleftrightarrow\; j = \rempn
  \end{equation*}
  holds.  Moreover, $\val_p(n + r -j) \neq 0$ if and only if
  $p \mid \rempn+r -j$, or equivalently, $j \equiv \rempn+r \mod p$. However,
  since $1\le r \le p-\rempn-1$, it follows that
  $1 + \rempn\le r + \rempn \le p-1$ and therefore $\val_p(n+r-j)\neq 0$ if
  and only if $j = \rempn +r$.
  
  Therefore, all entries of the first $\rempn$ columns of $\pblock{n}{p}$ are
  zero and the entries of the $\rempn+1$-st column equal
  $-\val_p(n-\rempn) \neq 0$ and the remaining columns
  $\rempn+1 \le j \le p-1$ are zero except for the entry
  in row $r = j-\rempn$ which amounts to $\val_p(n-\rempn)$. The assertion
  follows.
\end{proof}

Proposition~\ref{proposition:lefthalfpblock} describes the zero and
non-zero entries of the leftmost $p$ columns of the $p$-block
$\pblock{n}{p}$. Next, we have a closer look at the rightmost $p-1$
columns.

\begin{proposition}\label{proposition:righthalfpblock}
  Let $n\in\N$, $p \in \calP_n$ (cf.~Notation~\ref{notation:pr}) and
  $\pblock{n}{p} = (b_{r,j})$ the corresponding $p$-block.

  For $n-(p-1)\le j \le n-1$ and $1 \le r\in p-\rempn-1$ the
  following holds:
  \begin{equation*}
    b_{r,j} =
    \begin{cases}
      1 & j = n-p+r \\
      0 & \text{else}.
    \end{cases}
  \end{equation*}
\end{proposition}
A visualization of the statement in
Proposition~\ref{proposition:righthalfpblock} can be found in
Figure~\ref{fig:pblock-right}.

\begin{figure}[h]
  \centering
    \begin{tikzpicture}[mymatrixenv]
  \matrix[mymatrix] (m)  {
    &{}              &{}              &|[alias=mclro]|{} & |[alias=1b3o]|{1} & {}               &{}&|[alias=1b3or]|{} &{} &                {}&{}     \\
    \\
    &                &|[scale=1.5]|{*}&                  &                   & |[scale=1.5 ]|{} &  &                  &   &|[scale=1.5 ]|{0} & \\
    \\
    &|[alias=jPu]|{} &                &|[alias=mclru]|{} &|[alias=np1u]|{}   &                  &  &|[alias=1b3u]|1   &   &                  &|[alias=b1ur]|{}  \\
 \\
   };
   %
    \mymatrixbraceleft{1}{5}{\footnotesize $p-\rempn-1$}

    \node[above= 0.6cm of 1b3o] (j) {\footnotesize $j=n-(p-1)$};
    \draw[->, shorten >= 0.3cm] (j.south) -- (1b3o.north);

    \node[above=1.3cm of 1b3or] (p1) {\footnotesize $j = n-(\rempn+1)$};
    \draw[->, shorten >= 0.5cm] (p1.south) -- (1b3or.north);


    \mymatrixbracetop{9}{11}{\footnotesize $\rempn$}






    \draw[dotted,shorten <=0.1cm,shorten >=0.1cm] (1b3o.south east) -- (1b3u.north west);

    \draw[ForestGreen, rounded corners] ($(1b3o.north)+(-0.4em,0.4em)$) rectangle ($(b1ur)+(0.4em,-0.5em)$);

\end{tikzpicture}

    \caption{Rightmost $p-1$ columns of (the first $p-\rempn-1$ rows of) $\pblock{n}{p}$,
      cf.~Proposition~\ref{proposition:righthalfpblock}}
  \label{fig:pblock-right}
\end{figure}

\begin{proof}
  For $n-(p-1)\le j \le n-1$, it follows that $1 \le n-j \le p-1$ and
  therefore $\val_{p}(n-j) = 0$.  Moreover, $\val_p(n+r-j) \neq 0$ is
  equivalent to $p\mid n-j+r$. However, given the ranges of $n-j$ and
  $r$, this is the case if and only if $n-j+r = p$ in which case
  $\val_p(n+r-j) = 1$. The assertion follows.
\end{proof}

\subsection{The inner columns of \texorpdfstring{$\A_n$}{An}}
\label{subsection:inner}
In this section we prove the results concerning the ``inner'' columns
of $\A_n$, namely we show that the Goals~(ii) and~(iii) formulated in
Remark~\ref{remark:directsum} hold, see
Corollary~\ref{corollary:innerdim}.  Moreover, this leads to the
special case of Theorem~\ref{theorem:binomialpolynomials} where $n=P$
is a prime number, see Corollary~\ref{corollary:n=P} below.

In this section we exploit the structure of the $P$-block
$\pblock{n}{P}$ where $P=\max\calP_n$ is the maximal prime number
which is less than or equal to $n$.

\begin{remark}\label{remark:Pnoexception}
  Let $n\ge 2$ and $P=\max\calP_n$, cf.~Notation~\ref{notation:pr}. If
  $n = 2^s$ with $s\ge 2$, then $P>2$ and if $n=9$, then $P=7>3$. Therefore, 
  $\calR_{n,P} = \{r \mid 1\le r\le P-r_{n,P}-1\}$. 

  Hence, the $P$-block $\pblock{n}{P}$ always consists of
  $P-r_{n,P}-1$ rows and Propositions~\ref{proposition:lefthalfpblock}
  and~\ref{proposition:righthalfpblock} refer to the whole $P$-block,
  cf.~Remark~\ref{remark:exceptions}.
\end{remark}

\begin{proposition}\label{proposition:middle-entries}
  Let $n\ge 2$ be an integer and $P= \max\calP_n$ be the largest prime
  number which is less than or equal to $n$. 
  
  Then
  \begin{enumerate}
  \item $\dim\inspn(\pblock{n}{P}) = 2P-n-1$,
  \item $\outspn{0}{n-P-1}(\pblock{n}{P}) = \mathbf{0}$.
  \end{enumerate}
\end{proposition}
\begin{proof} 
  By Lemma~\ref{lemma:column-grouping}, $n-P < P-1$ holds.  In
  particular, it follows that $r_{n,P} = n-P$. Moreover,
  $\calR_{n,P} = \{1,2,\ldots, 2P-n-1\}$,
  cf.~Remark~\ref{remark:Pnoexception}.  By
  Propositions~\ref{proposition:lefthalfpblock}
  and~\ref{proposition:righthalfpblock} the $P$-block $\pblock{n}{P}$
  is of the following form, see also Figure~\ref{fig:Pblock}:
  \begin{enumerate}[(a)]
  \item All columns indexed with
    $j\in \{0,\ldots, n-P-1\} \cup \{P,\ldots, n-1\}$ are zero columns.
  \item Each entry of the column $j = r_{n,P} = n-P$ is equal $-1$.
  \item Each column indexed with $n-P+1 \le j \le P-1$ contains
    exactly one non-zero entry, namely, the entry in row $r = j+P-n$
    is equal $1$.
  \end{enumerate}
  \begin{figure}[h]
  \centering
    \begin{tikzpicture}[mymatrixenv]
  \matrix[mymatrix] (m)  {
 |[alias=b1o]|{}& {}             &{}&|[alias=mvb1o]|{-1}&|[alias=vb1o]|{1}&{}               & {}&|[alias=mclo]|{} &{}&{}               &{} \\
 {}             &  \\                                                                                         
 {}             &|[scale=1.5]|{0}&  &                   &                 & |[scale=1.5 ]|{}& {}&                 &  &|[scale=1.5]|{0} & \\
 {}             & \\                                                                                                             
 |[alias=b1u]|{}&                &  &|[alias=mvb1u]|{-1}&                 &                 & {}& |[alias=vb1u]|1 &  &                 & \\
 \\                                                                                         
   };
   %
    \mymatrixbraceleft{1}{5}{\footnotesize $2P-n-1$}

    \mymatrixbracetop{1}{3}{\footnotesize $n-P$}
    \mymatrixbracetop{9}{11}{\footnotesize $n-P$}
    \node[above= of mvb1o] (j) {\footnotesize $j=n-P$};
    \draw[->, shorten >= 0.3cm] (j.south) -- (mvb1o.north);

    \node[above=1cm of mclo] (p1) {\footnotesize $j = P-1$};
    \draw[->, shorten >= 0.55cm] (p1.south) -- (mclo.north);







    \draw[dotted,shorten <=0.1cm,shorten >=0.1cm] (mvb1o.south) -- (mvb1u.north);
    \draw[dotted,shorten <=0.1cm,shorten >=0.1cm] (vb1o.south east) -- (vb1u.north west);


    \draw[ForestGreen, rounded corners] ($(mvb1o.north)+(-0.8em,0.5em)$) rectangle ($(vb1u)+(0.4em,-0.7em)$);

\end{tikzpicture}

    \caption{$\pblock{n}{P}$}
  \label{fig:Pblock}
\end{figure}

  It immediately follows from (a), that
  $\outspn{0}{n-P-1}(\pblock{n}{P}) = \mathbf{0}$. Moreover, as
  $\inspn(\pblock{n}{P})$ is spanned by ${2P-n-1}$ vectors, the
  dimension is bounded by $2P-n-1$.  By (c), the submatrix of
  $\pblock{n}{P}$ consisting of columns $n-P+1 \le j \le P-1$ is the
  identity matrix of dimension $P-1 - (n-P+1)+1 = 2P-n-1$. As all
  these columns are elements of $\inspn(\pblock{n}{P})$ it follows
  that $\dim\inspn(\pblock{n}{P}) = 2P-n-1$.
\end{proof}

\begin{corollary}\label{corollary:innerdim}
  Let $n\ge 2$ be an integer and $P = \max\calP_n$ be the largest
  prime number less than or equal to $n$,
  cf.~Notation~\ref{notation:pr}.

  Then
  \begin{enumerate}[(i)]
  \item $\dim\inspn(\A_n) = 2P-n-1$ and
  \item $\outspn{0}{n-P-1}(\A_n) \cap \inspn(\A_n) = \mathbf{0}$.
  \end{enumerate}
\end{corollary}

Note that the assertions in Corollary~\ref{corollary:innerdim} are
exactly the Goals~(ii) and~(iii) of Remark~\ref{remark:directsum}.

\begin{proof}
  Since $\inspn(\A_n)$ is spanned by $2P-n-1$ vectors,
  $\dim\inspn(\A_n) \le 2P-n-1$ holds. Moreover, as the projection
  $\pi: \inspn(\A_n) \to \inspn(\pblock{n}{P})$ is an epimorphism, it
  follows that $\dim\inspn(\A_n) = 2P-n-1$ by
  Proposition~\ref{proposition:middle-entries}.
  
  For the second assertion, let
  $v \in \outspn{0}{n-P-1}(\A_n) \cap \inspn(\A_n)$. Let $a_{n-P+1}$,
  \ldots, $a_{P-1}$ denote the columns of $\A_n$ which span
  $\inspn(\A_n)$. Then there exist $\lambda_{n-P+1}$, \ldots,
  $\lambda_{P-1} \in \Q$ such that
  \begin{equation}\label{eq:Pblock-1}
    v = \sum_{j=n-P+1}^{P-1}\lambda_j a_j \in \outspn{0}{n-P-1}(\A_n) \cap \inspn(\A_n).
  \end{equation}
  Then
  \begin{equation}\label{eq:Pblock-2}
    \pi(v) = \sum_{j=n-P+1}^{P-1}\lambda_j b_j^{(P)} \in \outspn{0}{n-P-1}(\pblock{n}{P}) \cap \inspn(\pblock{n}{P})
  \end{equation}
  where $b_j^{(P)}$ denotes the $j$-th column of $\pblock{n}{P}$ for
  $n-P+1 \le j \le P-1$. By
  Proposition~\ref{proposition:middle-entries},
  $\outspn{0}{n-P-1}(\pblock{n}{P}) = \mathbf{0}$ and the columns
  $b_j^{(P)}$ are linearly independent. Hence $\pi(v) = \mathbf{0}$
  and Equation~\eqref{eq:Pblock-2} implies that $\lambda_j= 0$ for
  $n-P+1 \le j \le P-1$. Therefore, by plugging into
  Equation~\eqref{eq:Pblock-1}, it follows that $v=0$ which completes
  the proof.
\end{proof}

It follows from Corollary~\ref{corollary:innerdim} that the ``inner''
$2P-n-1$ columns of $\A_n$, that is, the columns which span $\inspn(\A_n)$
are linearly independent. This immediately implies the next corollary.

\begin{corollary}\label{corollary:ranklb}
  Let $n\ge 2$ be an integer and $\A_n$ its valuation matrix. 

  Then $\rank(\A_n) \ge 2P-n-1$.  
\end{corollary}

As shown below, the results so far imply the absolute irreducibility
of $\binom{x}{n}$ in the special case where $n=P$ is a prime
number. Note that this has already been shown by Frisch and
Nakato~\cite[Example~2.6]{Frisch-Nakato:2020:graph-theoretic} and our
proof only serves as an alternative.

\begin{corollary}\label{corollary:n=P}
  Let $P$ be a prime number, then $\binom{x}{P}$ is absolutely
  irreducible.
\end{corollary}
\begin{proof}
  If $n = P$, then $\rank(\A_P) \ge 2P-P-1 = P-1$ by
  Corollary~\ref{corollary:ranklb}. Since $\rank(\A_P) < P$ by
  Remark~\ref{remark:dimker>0}, it follows $\rank(\A_P) =
  P-1$. According to Proposition~\ref{proposition:kernelA},
  $\binom{x}{P}$ is absolutely irreducible.
\end{proof}

\subsection{The outer columns of \texorpdfstring{$\A_n$}{An}} For a
composite number $n\ge 2$, the span of the ``outer'' columns of $\A_n$
is not trivial.  The goal of this section is to show that
$\dim\outspn{0}{n-P-1}(\A_n) = 2(n-P)$, where (for the remainder of
this work) $P = \max\calP_n$ is the largest prime number which is less
than or equal to $n$ (cf.~Notation~\ref{notation:pr}), see
Proposition~\ref{proposition:outerdim}. This is the final ingredient
to prove that $\rank(\A_n) = n-1$ which then further implies that
$\binom{x}{n}$ is absolutely irreducible,
cf.~Remark~\ref{remark:directsum}.

As above, we have to exploit the structure of certain $p$-blocks to reach
this goal. In contrast to the arguments in the previous subsection, we
need to find more than one suitable $p$-block.  Which choices to make
is explained in detail in the proof of
Proposition~\ref{proposition:outerdim}. The next proposition gives information
about the dimension of $\outspn{\rempn}{k-1}(\pblock{n}{p})$ for certain choices
of $k$.

\begin{proposition}\label{proposition:spansofB}
  Let $n\ge 2$ be a composite number, $p \in \calP_n$
  (cf.~Notation~\ref{notation:pr}) be a prime number and
  $\pblock{n}{p}$ the corresponding $p$-block.
  
   If $\rempn+1\le k \le \min\{n-P, \frac{p+\rempn-1}{2}\}$, then
   \begin{enumerate}
   \item $\dim \outspn{\rempn}{k-1}(\pblock{n}{p}) = 2(k-\rempn)$ and
   \item $\outspn{0}{\rempn-1}(\pblock{n}{p})  = \mathbf{0}$.
   \end{enumerate}
\end{proposition}
\begin{proof}
  We first treat the special case $n=9$ and
  $p=3$. Figure~\ref{fig:3block9} displays~$\pblock{9}{3}$.
    \begin{figure}[h!]
      \centering
  \adjustbox{max width=5.3cm}{
    \begin{tikzpicture}[mymatrixenv]
      \node (d) at (-2.8,0) {$\pblock{9}{3}=$};
        \matrix[mymatrix] (m) {
             {-2} &  |[alias=mo]|{2} & |[alias=inspn]|{} & 1  & 0 \\
             {-2} &  0               & \phantom{aaaa}    & 0  & 1 \\
             {-1} &  0               &                   & 0  & 0 \\                   
             {-2} & |[alias=mu]|{1}  &                   & 1  & 0 \\
           };
        \node[above=0.4cm of inspn] (in) {\footnotesize $\inspn(\pblock{9}{3})$};
%
%
        \draw[Red, dashed]  ($(mo)+(+1em,0.6em)$) -- ($(mu)+(+1em, -0.65em)$);
        \coordinate (Wlo) at ($(mo)+(+1.4em,0.7em)$);
        \coordinate (Wlu) at ($(mu)+(+1.4em,-0.5em)$);
        \coordinate[right=0.7cm of Wlu] (Wru);
        \coordinate[right=0.7cm of Wlo] (Wro);
        \draw[fill=Gray, opacity=0.2] (Wlo)
        [snake=zigzag, segment length=10pt]-- (Wlu)
        [snake=zigzag, segment length=100pt]--
        (Wru) [snake=zigzag, segment
        length=10pt]--(Wro) [snake=zigzag, segment
        length=100pt]--(Wlo);
      \end{tikzpicture}.
    }
    \caption{The $3$-block $\pblock{9}{3}$ of $9$.}
    \label{fig:3block9}
  \end{figure}
  Since $r_{9,3}=0$, the second assertion of the proposition follows
  trivially. Moreover, since $1\le k \le \frac{3+0-1}{2}=1$, the first
  assertion can be verified by direct computation.
  
  For the remainder of the proof, we assume that $n\neq 9$ or
  $p\neq 3$.  It follows from
  $\rempn+1\le k \le \min\{n-P, \frac{p+\rempn-1}{2}\}$ that
  $\rempn \le p-3$ and $2k \le p+\rempn-1$. In particular, $p> 2$
  holds and it follows that the $p$-block $\pblock{n}{p}$ has
  $p-\rempn-1 \ge 2(k-\rempn)$ rows,
  cf.~Remark~\ref{remark:exceptions}.

  Moreover, since $k\le \frac{p+\rempn-1}{2} \le p-1$, we can apply
  Propositions~\ref{proposition:lefthalfpblock}
  and~\ref{proposition:righthalfpblock} to describe the outermost left
  $k$ and right $k$ columns of the (whole) $p$-block $\pblock{n}{p}$,
  depicted in Figure~\ref{fig:pblockrank}. They are of the following
  form:
  \begin{enumerate}[(a)]
  \item The columns indexed with
    $j\in \{0,\ldots, \rempn-1\} \cup \{n-\rempn,\ldots, n-1\}$ are zero
    columns (these are the $\rempn$ left outermost and $\rempn$ right
    outermost columns).
  \item The entries of column $j=\rempn$  are all equal $-\val_p(n-\rempn)\neq 0$.
  \item The columns $\rempn+1 \le j \le k-1$ have exactly one non-zero entry,
    $\val_p(n-\rempn)$, in row $r = j-\rempn$, that is, the first $k-(\rempn+1)$
    rows of these columns form a diagonal matrix with $\val_p(n-\rempn)$ on
    the diagonal and all entries below are zero.
  \item The columns $n-k \le j \le n-(\rempn+1)$ contain exactly one non-zero
    entry, namely 1, in row $r = j-(n-p)$, that is, the lower $k-\rempn$
    rows of these columns form the identity matrix and all entries
    above are zero.
  \item Since $k < \frac{p+\rempn}{2}$ by assumption, the first $k-(\rempn+1)$
    rows and the last $k-\rempn$ rows are disjoint and there is at
    least one row in the part in the middle, namely there is at least
    one row indexed with $k-\rempn \le r \le p-k-1$.
  \end{enumerate}
  \begin{figure}[h!]
    \centering
    \begin{adjustbox}{width=\textwidth}
      \input{Bcutk.tex}
    \end{adjustbox}
    \caption{$p$-block $\pblock{n}{p}$ with $p+\rempn>2k$}
    \label{fig:pblockrank}
  \end{figure}

  It immediately follows that
  $\outspn{0}{\rempn-1}(\pblock{n}{p}) = \mathbf{0}$ since it is
  spanned only by zero columns. It remains to show that
  $\dim \outspn{\rempn}{k-1}(\pblock{n}{p}) = 2(k-\rempn)$.

  Let $b_j$ denote the $j$-th column of $\pblock{n}{p}$ and assume that there
  are $\lambda_j\in \Q$ with
  $j\in \{\rempn,\ldots, k-1\} \cup \{n-k,\ldots, n-(\rempn+1)\}$ such that
  \begin{equation}\label{eq:spanintersection}
    \sum_{j=\rempn}^{k-1}\lambda_jb_j + \sum_{j=n-k}^{n-(\rempn+1)}\lambda_jb_j = 0
  \end{equation}

  Since $b_{\rempn}$ is the only column which has a non-zero entry in
  the rows indexed with $k-\rempn \le r \le p-k-1$,
  Equation~\eqref{eq:spanintersection} implies that
  $\lambda_{\rempn}=0$.  Now for each of the remaining rows, there is
  exactly one column with a non-zero entry in this (and no other)
  row. With the same reasoning we can conclude that $\lambda_j = 0$
  for $ \rempn+1 \le j \le k-1$ and $n-k\le j\le n-(\rempn+1)$.
  Therefore, these $2(k-\rempn)$ columns of $\pblock{n}{p}$ are
  linearly independent.
\end{proof}

\begin{proposition}\label{proposition:outerdim}
  Let $n\ge 2$ be a composite integer and $\A_n$ its valuation matrix. 

  Then $\dim\outspn{0}{n-P-1}(\A_n) = 2(n-P)$.  
\end{proposition}

\begin{proof}
  We treat the cases $n=9$ and $n=10$ first as we want to exclude them
  below. For $n=9$, observe that $P=7$ and hence $n-P=2$.  It follows
  from a direct verification that the outer four columns of
  $\pblock{9}{3}$ (displayed in Figure~\ref{fig:3block9}) are linearly
  independent and hence $\dim\outspn{0}{1}(\pblock{9}{3}) = 4$. Since
  the projection
  $\outspn{0}{1}(\A_9) \to \outspn{0}{1}(\pblock{9}{3})$ is an
  epimorphism, it follows that $\dim\outspn{0}{1}(\A_9) = 4$.

  The valuation matrix $\A_{10}$ is displayed in
  Figure~\ref{fig:10block}. A direct computation verifies that
  $\rank(\outspn{0}{2}(\A_{10})) = 6$.
  \begin{figure}[h!]
    \centering
    \begin{tikzpicture}[mymatrixenv]
        \node (d) at (-4.1,0) {$\A_{10} = $};
        \matrix[mymatrix] (m) {
             {-1} &  1               & |[alias=mo]|{-3} &|[alias=inspn]|{}& 2  & -1 & {1} \\
              {}  &                  &                  &                 &    &    & {}  \\                   
              {0} & -2               &               2  &\phantom{aaaa}   & -1 & 1  & {0} \\
              {}  &                  &                  &                 &    &    & {}  \\                   
             {-1} & 1                &               0  &\phantom{aaaa}   & 0  & 0  & {0}  \\
             {-1} & 0                &               1  &\phantom{aaaa}   & 1  & 0  & {0}  \\
             -1   & 0                &               0  &\phantom{aaaa}   & 0  & 1  & {0}  \\
             -1   & 0                &    0             &\phantom{aaaa}   & 0  & 0  & {1}  \\
              {}  &                  &                  &                 &    &    &  {}   \\                   
              {}  & |[scale=1.5]|{0} &|[alias=mu]|{}    &                 &    &|[scale=1.5]|{0} & {}\\
           };
           \node[above=0.4cm of inspn] (in) {\footnotesize $\inspn(\A_{10})$};

           \mymatrixbraceright{1}{1}{\footnotesize $\pblock{10}{2}$}
           \mymatrixbraceright{3}{3}{\footnotesize $\pblock{10}{3}$}
           \mymatrixbraceright{5}{8}{\footnotesize $\pblock{10}{5}$}
           \draw[mymatrixbrace] ($(m-10-7.north east)+(16pt,10pt)$)
           -- node[right=2pt] {\footnotesize $\pblock{10}{7}$}        ($(m-10-7.south east) + (16pt,-3pt)$);
        
        \coordinate (Wlo) at ($(mo)+(+1.9em,0.7em)$);
        \coordinate (Wlu) at ($(mu)+(+1.9em,-0.3em)$);
        \coordinate[right=0.7cm of Wlu] (Wru);
        \coordinate[right=0.7cm of Wlo] (Wro);

        \draw[fill=Gray, opacity=0.2] (Wlo)
        [snake=zigzag, segment length=10pt]-- (Wlu)
        [snake=zigzag, segment length=100pt]--
        (Wru) [snake=zigzag, segment
        length=10pt]--(Wro) [snake=zigzag, segment
        length=100pt]--(Wlo);

        \draw[Red, dashed]  ($(mo)+(+1.3em,0.6em)$) -- ($(mu)+(+1.3em, -0.4em)$);
        
        \draw[ForestGreen, rounded corners] ($(m-1-1)+(-0.8em,0.7em)$) rectangle ($(m-1-7)+(0.5em,-0.7em)$);
        \draw[ForestGreen, rounded corners] ($(m-3-1)+(-0.8em,0.7em)$) rectangle ($(m-3-7)+(0.5em,-0.7em)$);
        \draw[ForestGreen, rounded corners] ($(m-5-1)+(-0.8em,0.7em)$) rectangle ($(m-8-7)+(0.5em,-0.7em)$);
        \draw[ForestGreen, rounded corners] ($(m-10-1)+(-0.8em,1.2em)$) rectangle ($(m-10-7)+(0.5em,-0.3em)$);
    \end{tikzpicture}.
    \caption{$\outspn{0}{2}(\A_{10})$ has dimension $6$.} 
    \label{fig:10block}
  \end{figure}

  From now on, assume that $n\ge 2$ is a composite integer with
  $n\neq 9$ and $n\neq 10$.  We prove by induction on $1\le k \le n-P$
  that
  \begin{equation*}
    \dim \outspn{0}{k-1}(\A_n) = 2k
  \end{equation*}
  holds.

  \textbf{Base case $k=1$}.  Since the projection
  $\outspn{0}{0}(\A_n) \to \outspn{0}{0}(\pblock{n}{p})$ is an
  epimorphims, it follows that
  $2 \ge\dim\outspn{0}{0}(\A_n) \ge \dim\outspn{0}{0}(\pblock{n}{p})$
  for all $p\in \calP_n$. Therefore, it suffices to show that there
  always exists a prime number $p$ with
  $\dim\outspn{0}{0}(\pblock{n}{p}) = 2$.

  If $n = 2^s$ is a proper power of $2$, then
  $\calR_{2^s,2} = \{1,2\}$ by definition which implies that
  $\pblock{2^s}{2}$ has two rows. The outermost columns of
  $\pblock{2^s}{2}$ are displayed in Figure~\ref{fig:2block-proof}.
  \begin{figure}[h]
    \centering
      \begin{tikzpicture}[mymatrixenv]
        \node (d) at (-3.4,0) {$\pblock{2^s}{2}=$};
       \matrix[mymatrix] (m) {
          |[alias=mo]|{-\val_2(n)} & * &|[alias=inspn]|{\phantom{aaaa}}& *  & 1 \\
          |[alias=mu]|{1-\val_2(n)}& * &                               & *  & 0 \\
        };
        \node[above=0.2cm of inspn] (in) {\footnotesize $\inspn(\pblock{2^s}{2})$};

        \draw[Red, dashed]  ($(mo)+(+4.5em,0.7em)$) -- ($(mu)+(+4.5em, -0.75em)$);
        
        \coordinate (Wlo) at ($(mo)+(+5em,0.7em)$);
        \coordinate (Wlu) at ($(mu)+(+5em,-0.7em)$);
        \coordinate[right=0.5cm of Wlu] (Wru);
        \coordinate[right=0.5cm of Wlo] (Wro);

        \draw[fill=Gray, opacity=0.2] (Wlo)
        [snake=zigzag, segment length=10pt]-- (Wlu)
        [snake=zigzag, segment length=100pt]--
        (Wru) [snake=zigzag, segment
        length=10pt]--(Wro) [snake=zigzag, segment
        length=100pt]--(Wlo);
      \end{tikzpicture}
    \caption{$\pblock{2^s}{2}$}
    \label{fig:2block-proof}
  \end{figure}
  As $\val_2(n)=s \ge 2$, it follows that
  $\dim\outspn{0}{0}(\pblock{2^s}{2}) = 2$.
  
  If $n$ is not a power of $2$, then $n$ has a prime divisor $p\ge
  3$. It follows that $\rempn=0$ and hence
  $1 = \rempn + 1 \le k=1 \le \min\{n-P,\frac{p-1}{2}\}$. Therefore, we can
  apply Proposition~\ref{proposition:spansofB} and conclude that
  $\dim\outspn{0}{0}(\pblock{n}{p}) = 2$.  This completes the proof of
  the base case.

  Observe that the induction base case also covers the cases $n=4$,
  $n=6$ and $n=8$ because in each of these cases $n-P = 1$. Therefore,
  we already handled all cases where $n\le 10$.

  We proceed with the \textbf{induction step} and prove the following claim.
  
  \noindent\textbf{Claim}. There exists $p\in \calP_n$ with $\rempn \le k-1$ such
  that
  \begin{enumerate}
  \item\label{rankprf-1} $\dim \outspn{\rempn}{k-1}(\A_n) = 2(k-\rempn)$ and
  \item\label{rankprf-2} $\outspn{0}{\rempn-1}(\A_n) \cap \outspn{\rempn}{k-1}(\A_n) = \mathbf{0}$. 
  \end{enumerate}
  Assume for a moment that this claim holds. Since $\rempn < k$ we can
  apply the induction hypothesis, that is,
  $\dim \outspn{0}{\rempn-1}(\A_n) = 2\rempn$ and hence
  \begin{align*}
    \dim \outspn{0}{k-1}(\A_n) &= \dim \left(\outspn{0}{\rempn-1}(\A_n) \oplus \outspn{\rempn}{k-1}(\A_n)\right) \\
    &= \dim \outspn{0}{\rempn-1}(\A_n) + \dim\outspn{\rempn}{k-1}(\A_n) \\
                            &= 2\rempn + 2(k-\rempn) \\
                            &= 2k
  \end{align*}
  and the assertion follows.

  Next, we prove the claim. By Theorem~\ref{theorem:numbertheory} (for
  which we needed to assume that $n> 10$), there exists $p\in \calP_n$
  such that $p > 2k$ and $\rempn\le k-1$.

  Therefore
  $\rempn+1 \le k \le \min\{n-P, \frac{p-1}{2}\} \le \min\{n-P,
  \frac{p+\rempn-1}{2}\}$ and hence, by
  Proposition~\ref{proposition:spansofB},
  $\dim \outspn{\rempn}{k-1}(\pblock{n}{p}) = 2(k-\rempn)$ holds.
  
  Let
  $\pi: \outspn{\rempn}{k-1}(\A_n) \to
  \outspn{\rempn}{k-1}(\pblock{n}{p})$ be the projection. As $\pi$ is
  an epimorphism, it follows that
  $\dim \outspn{\rempn}{k-1}(\A_n) \ge \dim \outspn{\rempn}{k-1}(\pblock{n}{p})$.  Hence
  \eqref{rankprf-1} of the claim above follows.

  To prove \eqref{rankprf-2} of the claim, assume that
  $v \in \outspn{0}{\rempn-1}(\A_n) \cap \outspn{\rempn}{k-1}(\A_n)$
  and let
  $\calJ_0 = \{j \mid 0\le j \le \rempn-1 \text{ or } n-\rempn \le j
  \le n-1\}$ and
  $\calJ_1=\{j \mid \rempn \le j \le k-1 \text{ or } n-k\le j \le
  n-(\rempn+1) \}$ be the indices of the columns of $\A_n$ which span
  $\outspn{0}{\rempn-1}(\A_n)$ and $\outspn{\rempn}{k-1}(\A_n)$,
  respectively.

  For $0 \le j \le n-1$, let $a_{j}$ denote the $j$-th column of $\A_n$
  and $\lambda_j\in \Q$ such that
  \begin{equation}\label{eq:pbl-1}
    v = \sum_{j\in \calJ_0}\lambda_j a_j  = \sum_{j\in \calJ_1}\lambda_j a_j\in \outspn{0}{\rempn-1}(\A_n) \cap \outspn{\rempn}{k-1}(\A_n).
  \end{equation}
  Then
  \begin{equation}\label{eq:pbl-2}
    \pi(v) = \sum_{j\in \calJ_0}\lambda_j b_j^{(p)} = \sum_{j\in\calJ_1}\lambda_j b_j^{(p)}\in \outspn{0}{\rempn-1}(\pblock{n}{p}) \cap \outspn{\rempn}{k-1}(\pblock{n}{p})
  \end{equation}
  where $b_j^{(p)}$ denotes the $j$-th column of
  $\pblock{n}{p}$. However, by Proposition~\ref{proposition:spansofB},
  $\outspn{0}{\rempn-1}(\pblock{n}{p}) = \mathbf{0}$ and hence
  Equation~\eqref{eq:pbl-2} reduces to
  \begin{equation*}
    0 = \pi(v) =  \sum_{j\in \calJ_1}\lambda_j b_j^{(p)}.
  \end{equation*}
  As the columns $b_j^{(p)}$ with $j\in\calJ_1$ of $\pblock{n}{p}$ are
  linearly independent by Proposition~\ref{proposition:spansofB}, it
  follows that $\lambda_j= 0$ for all $j\in \calJ_1$. Therefore,
  plugging into Equation~\eqref{eq:pbl-1}, it follows that $v = 0$
  which completes the proof of the claim.

\end{proof}

Proposition~\ref{proposition:outerdim} together with
Corollary~\ref{corollary:innerdim} imply that $\rank(\A_n) = n-1$ for
all composite numbers $n\ge 2$,
cf.~Remark~\ref{remark:directsum}. Using
Proposition~\ref{proposition:kernelA} yields the corollary below.

\begin{corollary}\label{corollary:ncomposite}
  Let $n\ge 2$ be a composite number.

  Then $\binom{x}{n}$ is absolutely irreducible.
\end{corollary}

\renewcommand{\MR}[1]{}
\bibliography{bibliography}
\bibliographystyle{amsplainurl}

\end{document}